\documentclass[12pt, reqno]{amsart}
\usepackage{amssymb, graphicx}
\usepackage{amsmath,amsthm,amssymb,mathrsfs,amsfonts}
\usepackage{hyperref}
\usepackage{a4wide}
\usepackage{xcolor}
\usepackage{float}
\usepackage{pgf,tikz}
\usetikzlibrary{arrows}

\def \D {\mathbb{D}}

\def \T {\mathbb{T}}

\newcommand{\cD}{{\mathcal{D}}}
\newcommand{\cZ}{{\mathcal{Z}}}

\newtheorem{theorem}{Theorem}[section]
\newtheorem{lemma}[theorem]{Lemma}

\newtheorem{remark}[theorem]{Remark}

\newtheorem{corollary}[theorem]{Corollary}

\begin{document}

\title[Extremal functions and zero sets for the Dirichlet space]{Extremal functions and zero sets for the Dirichlet space}
\author[A. Borichev]{{Alexander}{ Borichev}}
\author[O. El-Fallah]{{Omar}{ El-Fallah}}
\author[K. Kellay]{{Karim}{ Kellay}}
\subjclass[2000]{Primary 46E22; Secondary 31A05, 31A15, 31A20, 47B32}
\keywords{Dirichlet  spaces, zero sets, extremal function, reproducing kernel}
\address{A.~Borichev, Aix-Marseille University, CNRS, I2M, Marseille, France}
\email{alexander.borichev@math.cnrs.fr}

\address{O.~El-Fallah, Laboratory of Mathematical Analysis and Applications \\  Mohamed V University in Rabat \\  B.P. 1014 Rabat\\Morocco}
\email{o.elfallah@um5r.ac.ma}

\address{K.~Kellay, University of Bordeaux, CNRS, Bordeaux INP, IMB, UMR 5251, F-33400, Talence, France}  
\email{kkellay@math.u-bordeaux.fr}

\begin{abstract}
We study the zeros of functions in the Dirichlet space. Using extremal functions, we produce a necessary and sufficient condition for a sequence of points in the unit disk to be a zero set of the classical Dirichlet space. This Shapiro--Shields type condition involves kernels of the harmonic Dirichlet space associated with measures depending on the zero sequence.
\end{abstract}

\maketitle

\section{Introduction}

Let $H^2$ denote the classical {\it Hardy space} of analytic functions on the unit disk $\D$ having square summable Taylor coefficients at the origin. Every function $f \in H^2$ has non-tangential limits almost everywhere on the unit circle $\T = \partial \D$. We denote by $f(\zeta)$ the non-tangential limit of $f$ at $\zeta \in \T$, if it exists.

Let $\mu$ be a positive finite measure on $\T$. The so called {\it harmonic Dirichlet space} $\cD(\mu)$ is the set of analytic functions $f \in H^2$ such that
$$
\cD_{\mu}(f) := \frac{1}{2\pi} \int_{\T} \int_{\T} \frac{|f(\zeta) - f(\xi)|^2}{|\zeta - \xi|^2} \,|d\zeta| \,d\mu(\xi) < \infty.
$$
By Douglas' formula (see, for instance, \cite[Proposition 2.2]{RS91}), 
$$
\cD_{\mu}(f) = \int_{\D} |f'(z)|^2 P_\mu(z) \, dA(z),
$$
where $dA(z)$ is the normalized Lebesgue measure on $\D$ and $P_\mu$ is the Poisson integral of $\mu$ given by
$$
P_\mu(z) = \int_\T \frac{1 - |z|^2}{|1 - \bar{\zeta} z|^2} \,d\mu(\zeta), \qquad z \in \D.
$$
The space $\mathcal{D}(\mu)$ is endowed with the norm
$$
\|f\|_{\mathcal{D}(\mu)}^{2} := \|f\|^{2}_{H^2} + \cD_{\mu}(f).
$$
When $\mu=m$, the normalized arc measure on $\T$, $\cD(\mu)$ reduces to the classical Dirichlet space $\cD$, consisting of all
holomorphic functions $f$ on the unit disk $\D$ whose derivative $f'$ is square-integrable. S.~Richter introduced in  \cite{Ri1,R} harmonic Dirichlet spaces, motivated by his study of cyclic analytic two-isometries and shift invariant subspaces on the Dirichlet space. S.~Richter and C.~Sundberg extended in \cite{RS} this study to deal with 

the lattice of closed shift invariant subspaces of $\mathcal{D}(\mu)$. However, these results do not readily yield a description of the cyclic vectors for the shift on $\mathcal{D}(\mu)$ nor the characterization of the zero set for Dirichlet space functions.

Denote by $k^\mu_z$ the reproducing kernel of $\cD(\mu)$ at $z \in \mathbb{D}$, 
$$
f(z) = \langle f, k^\mu_z \rangle_{\mathcal{D}(\mu)}, \qquad f \in \cD(\mu).
$$
When $\mu = 0$, $\cD(\mu)$ is the Hardy space $H^2$, and, hence, the reproducing kernel $k^0_z$ is the Cauchy kernel.  If $\mu=m$, the normalized Lebesgue measure on the unit circle,  then $\mathcal{D}(\mu) = \mathcal{D}$, and the reproducing kernel of the Dirichlet space is given by
\begin{equation}
\label{kerneldirichet}
k^m_w(z) = \frac{1}{\overline{w}z} \log \left( \frac{1}{1 - \overline{w}z} \right), \qquad z, w \in \D.
\end{equation}

Let us mention that H.~Shapiro--A.~Shields \cite{SS} 
improved an earlier result of L.~Carleson~\cite{C} by showing that the condition 
\[
\sum_{z\in \mathcal{Z}} 
\frac{1}{\log (1/(1-|z|))} < \infty
\label{1.1}
\]
is sufficient for $\cZ$ to be a zero set for $\mathcal D$. 

S.~Shimorin proved in \cite[Theorem 1.1]{Shi2} that all harmonic Dirichlet spaces $\cD(\mu)$ have complete Nevanlinna--Pick reproducing kernels. As an important consequence (see 
\cite[Chapter 2, Theorem 1]{Seip}), every 
sequence $\cZ  
\subset \D$ satisfying the Shapiro--Shields condition
\begin{equation}\label{SScondition}
\sum_{z \in \cZ} {1}/{k^\mu_z(z)} < \infty,
\end{equation}
is a zero set of $\cD(\mu)$. 

An asymptotic estimate of the reproducing kernel $k^\mu$ of $\cD(\mu)$ on the diagonal was given in \cite{EEK} (see Remark~\ref{r32} below). 

It is known that condition \eqref{SScondition} is not necessary in the general case and, in particular, for the classical Dirichlet space, see \cite{KM, NRS, RRS}. 

In the case of the Hardy space, \eqref{SScondition} becomes the Blaschke condition
$$
\sum_{z \in \cZ} (1 - |z|) < \infty,
$$
which is a necessary and sufficient condition for the sequence $\cZ$ to be a zero set of $H^2$. 
On the other hand, if $B$ is a Blaschke product associated with the set $\cZ$, then $\cZ$ is not a zero set for the model space $K_B=H^2\ominus BH^2$ although
$\sum_{z\in \cZ}1/k_z(z)=\sum_{z\in \cZ}(1-|z|)<\infty$, where $k$ is the reproducing kernel in $K_B$.
Model spaces fail to have the complete Nevanlinna--Pick property, except when they are one-dimensional \cite{CC}.

In this paper, we first establish a necessary and sufficient condition for a sequence of points in the unit disk to be a zero set of the classical Dirichlet space. This condition is of Shapiro-Shields type and involves so called extremal functions in the Dirichlet spaces associated with measures depending on the sequence, see Section~\ref{sect2} below. As a corollary, we then present some sufficient conditions for the zero sets, under additional separation conditions. 

For an alternative approach to the Dirichlet spaces zero sets, also using the extremal functions, see the recent paper \cite{HR}. For the Dirichlet space and its associated properties, we refer the reader to \cite{ARSW,EKMR}.

\section{Background on extremal functions}\label{sect2}

 A function $\phi \in \cD(\mu)$ is said to be extremal if 
\[
 \| \phi \|_{\cD(\mu)} = 1 \quad \text{ and } \quad \langle \phi, z^n\phi\rangle_{\cD(\mu)} = 0,  \quad n \geq 1.
\] 
Any extremal function is a multiplier, see, for instance, \cite[Theorem 3.1]{RS92}. A.~Aleman \cite[Lemma 4.8]{A}, and S.~Shimorin \cite[Theorem 1]{Shi1}, proved that extremal functions $\phi \in \cD(\mu)$ satisfy 
$$
|\phi(z)| \le 1, \qquad z \in \mathbb{D}.
$$

Given any non-zero closed subspace $\mathcal{M}$ of $\cD(\mu)$ invariant under the shift operator $M_z:f\mapsto zf$, we have $\dim(\mathcal{M}\ominus \mathcal{M})=1$, see, for instance, \cite[Theorem 3.2]{RS92}. 
If $\phi_\mathcal{M}\in\mathcal{M}\ominus \mathcal{M}$ and $\|\phi_\mathcal{M}\|_\mu=1$, then $\phi_\mathcal{M}$ is an extremal function for $\cD(\mu)$. 
In this situation, we say that $\phi_\mathcal{M}$ is an extremal function associated with $\mathcal{M}$.

From now on, we deal with absolutely continuous measures $\mu$. Given $f\in \cD(\mu)$, we define $d\mu_f(\zeta)= |f(\zeta)|^2 d\mu(\zeta)$, with $f(\zeta)$ being the radial limits. Then $\mu_f$ is also absolutely continuous. If $f$ is an extremal function for $\cD(\mu)$, then the radial limits of $f$ exist on the whole unit circle, see  \cite[Theorem 5.1]{RS}.

S.~Richter \cite[Theorem 7.1]{R91} proved that for every $M_z$-invariant subspace $\mathcal{M}$ of $\cD(\mu)$ we have  
$$
\mathcal{M} = \phi_\mathcal{M} \cD(\mu_{\phi_\mathcal{M}})
$$
and 
$$
\|f\phi_\mathcal{M}\|_{\cD(\mu)} = \|f\|_{\cD(\mu_{\phi_\mathcal{M}})}, \qquad f \in \cD(\mu_{\phi_\mathcal{M}}),
$$
where $\phi_\mathcal{M}$ is an extremal function associated with $\mathcal{M}$.

Furthermore, it is easily seen that if $\phi\in\cD(\mu)\setminus\{0\}$ and 
$$
\|f\phi\|_{\cD(\mu)} = \|f\|_{\cD(\mu_\phi)},\qquad f \in \cD(\mu_\phi),
$$ 
then $\phi$ is an extremal function for $\cD(\mu)$. 

This implies that the extremality property is related to the multiplicative structure, and we obtain the following lemma.

\begin{lemma}\label{produit}
   If $\phi$ is an extremal function for $\cD(\mu)$ and $\psi$ is an extremal function for $\cD(\mu_\phi)$, then $\phi\psi$ is an extremal function for $\cD(\mu)$.
\end{lemma}

\begin{proof}
Let \( f \in \cD(\mu_{\phi\psi}) \). Since 
$$ \|f\phi\psi\|_{\cD(\mu)} = \|f\psi\|_{\cD(\mu_\phi)} = \|f\|_{\cD(\mu_{\phi\psi})}, $$
we obtain the desired result.
\end{proof}

Let $n \geq 1$ and let
\[
\mathcal{Z}_n = ( z_1, z_2, \ldots, z_n )\subset \mathbb{D}.
\]
Denote by $\mathcal{M}(\mathcal{Z}_n,\mathcal{D}(\mu))$ the invariant subspace of
$\mathcal{D}(\mu)$ consisting of all functions vanishing on $\mathcal{Z}_n$, that is,
\[
\mathcal{M}(\mathcal{Z}_n,\mathcal{D}(\mu))
= \{ f \in \mathcal{D}(\mu) : f|_{\mathcal{Z}_n} = 0 \}.
\]
We have
\[
\mathcal{M}(\mathcal{Z}_n,\mathcal{D})
= \phi_n\, \mathcal{D}(\mu_n),
\]
where $\phi_n$ is an extremal function associated with
$\mathcal{M}(\mathcal{Z}_n,\mathcal{D})$ and
\[
d\mu_n = |\phi_n|^2 \, dm .
\]

Using Lemma~\ref{produit}, the function $\phi_n$ can be expressed as the product of
extremal functions, each vanishing at a single point of $\mathcal{Z}_n$.
Indeed,
\[
\mathcal{M}(\{z_1\},\mathcal{D})
= \phi_1\, \mathcal{D}(\mu_1)
= \phi_{z_1,m}\, \mathcal{D}(\mu_1),
\]
where $\phi_1 = \phi_{z_1,m}$ is the extremal function associated with
$\mathcal{M}(\{z_1\},\mathcal{D})$ and
\[
d\mu_1 = |\phi_1(\zeta)|^2 \, dm .
\]
Next,
\[
\mathcal{M}(\{z_1,z_2\},\mathcal{D})
= \phi_2\, \mathcal{D}(\mu_2)
= \phi_{z_2,\mu_1}\, \phi_1\, \mathcal{D}(\mu_2),
\]
where $\phi_2$ is an extremal function associated with
$\mathcal{M}(\mathcal{Z}_2,\mathcal{D})$,
$d\mu_2 = |\phi_2|^2 \, dm$, and $\phi_{z_2,\mu_1}$ is an extremal function
associated with $\mathcal{M}(\{z_2\},\mathcal{D}(\mu_1))$.
By Lemma~\ref{produit}, the product $\phi_{z_2,\mu_1}\phi_1$ is an extremal
function for $\mathcal{D}(\mu_2)$.

Finally, we obtain the following lemma, which describes a multiplicative property of extremal functions associated with the zero sequence $\mathcal{Z}_n$.

\begin{lemma} \label{fact2}
We have 
$$
\mathcal{M}(\mathcal{Z}_n, \cD(\mu)) = \phi_n \mathcal{D}(\mu_n) = \biggl( \prod_{j=1}^{n} \phi_{z_j, \mu_{j-1}} \biggr) \mathcal{D}(\mu_n),
$$
where \( d\mu_j = |\phi_j(\zeta)|^2 \, dm \), \( \phi_j \) is an extremal function for \( \mathcal{M}(\mathcal{Z}_j, \mathcal{D}) \), and \( \phi_{z_j, \mu_{j-1}} \) is an extremal function for \( \mathcal{M}(\{z_j\}, \mathcal{D}(\mu_{j-1})) \) for \( j = 1, \ldots, n \), with \( \mu_0 = m \).
\end{lemma}

Recall that an extremal function \( \phi_{a,\mu} \) for \( \mathcal{M}(\{a\}, \mathcal{D}(\mu)) \) is given by
\begin{equation}\label{extr1}
\phi_{a,\mu}(z) = \Bigl( 1 - \frac{1}{k^{\mu}_a(a)} \Bigr)^{-1/2} \Bigl( 1 - \frac{k^{\mu}_a(z)}{k^{\mu}_a(a)} \Bigr),
\end{equation}
where \( k^{\mu}_a(z) \) is the reproducing kernel for \( \mathcal{D}(\mu) \), and that $k^{\mu}_0(z) =k^{\mu}_z(0) = 1$, see \cite[Sections 2, 3]{Shi1}.

\section{Necessary and sufficient condition for Dirichlet zero sequences} 

\begin{theorem}\label{ShapiroShieldsdmu} 
Let \( \mathcal{Z} = (z_n)_{n \geq 1} \subset \mathcal{D} \). The sequence \( \mathcal{Z} \) is a zero set for \( \mathcal{D} \) if and only if 
\begin{equation}\label{SScond}
\sum_{n\geq 1} 1/{k^{\mu_{n-1}}_{z_n}(z_n)} < \infty, 
\end{equation}
where $d\mu_n=|\phi_n(\zeta)|^2\,dm$, $\phi_n$ is an extremal function for $\mathcal{M}_n=\mathcal{M}(\mathcal{Z}_n, \mathcal{D})$, $n\ge 1$, and $\mu_0=m$.
\end{theorem}

\begin{proof}
Without loss of generality, we can assume that $0\not\in \mathcal{Z}$. 
By Lemma \ref{fact2} and by formula \eqref{extr1}, we can express (unimodular multiples of) $\phi_n$ as  
$$
\phi_n(z) = \prod_{j=1}^{n} \phi_{z_j, \mu_{j-1}}(z) = \prod_{j=1}^{n} \Bigl( 1 - \frac{1}{k^{\mu_{j-1}}_{z_j}(z_j)} \Bigr)^{-1/2} \Bigl( 1 - \frac{k^{\mu_{j-1}}_{z_j}(z)}{k^{\mu_{j-1}}_{z_j}(z_j)} \Bigr).
$$ 
Hence, at $z=0$, we have
$$
\phi_n(0) = \prod_{j=1}^{n}  \Bigl( 1 - \frac{1}{k^{\mu_{j-1}}_{z_j}(z_j)} \Bigr)^{1/2}.
$$

Since $(\mathcal{M}_n)_{n\ge1}$ is a decreasing family of shift-invariant subspaces of $\cD(\mu)$, by \cite[Proposition 5.2]{EEL}, if 
$\mathcal{M}=\bigcap_{n\ge1}(\mathcal{M}_n)\not=\{0\}$, then the sequence $(\phi_n)_{n\ge 1}$ converges weakly to $\phi_{\mathcal{M}}$. Hence, $\lim_{n\to\infty}\phi_n(0)>0$, and we obtain \eqref{SScond}.

In the opposite direction, under condition \eqref{SScond}, we have $\lim_{n\to\infty}\phi_n(0)>0$. By normality and passing, if necessary, to a subsequence, we can find $\phi\in\cD(\mu)$ vanishing on $\mathcal{Z}$ such that $\phi(0)\not=0$. 
\end{proof}

\begin{remark}  \label{r32}
To apply criterion \eqref{SScond}, one needs to efficiently estimate the reproducing kernel \( k^\mu \) of \( \cD(\mu) \) on the diagonal. One can use here 
the asymptotic estimate
$$
k^{\mu}_{z}(z) \asymp 1 + \int_{0}^{|z|} \frac{dr}{(1 - r) P_\mu(rz/|z|) + (1 - r)^2}, \qquad z \in \D,
$$
with the implied constants independent of $z$ and $\mu$, see \cite[Theorem 1]{EEK}. 
\end{remark}  

\begin{remark}  
A.~Aleman considered in \cite{A} a more general class of Dirichlet spaces determined by 
positive superharmonic weights, which includes the \( \cD(\mu) \) spaces. 
Given a finite positive Borel measure \( \mu \) on the closed unit disk \( \overline{\D} \), we associate to it the Dirichlet space \(\widetilde{\mathcal{D}}(\mu) \) defined as the space of analytic functions \( f \) on \( \D \) satisfying the property
$$
\frac{1}{2\pi} \int_{\T} \int_{\overline{\D}} \frac{|f(\zeta) - f(z)|^2}{|\zeta - z|^{2}} \, |d\zeta| \, d\mu(z) < \infty.
$$
This class of spaces \( \widetilde{\mathcal{D}}(\mu) \) includes the important weighted Dirichlet spaces \( \mathcal{D}_\alpha \) (\( 0 \leq \alpha \leq 1 \)) defined as the spaces of analytic functions \( f \) on \( \D \) satisfying 
$$
\int_\D |f'(z)|^2 (1 - |z|^2)^\alpha \, dA(z) < \infty.
$$
In fact, $\mathcal{D}_\alpha=\widetilde{\mathcal{D}}(\mu)$ with 
$$
d\mu(z) = -(1 - |z|^2) \partial \overline{\partial} (1 - |z|^2)^\alpha \, dA(z), \qquad z \in \D.
$$
The results of A.~Aleman (see \cite[Theorem 4.9]{A}) generalize those of S.~Richter for \( \mathcal{D}(\mu) \) to the class $\widetilde{\mathcal{D}}(\mu)$. As a consequence, 
Theorem~\ref{ShapiroShieldsdmu} extends to the weighted Dirichlet spaces \( \widetilde{\mathcal{D}}(\mu) \). Note that if \( \alpha = 1 \), then \( \mathcal{D}_1 = H^2 \), 
and \eqref{SScond} becomes the Blaschke condition.
\end{remark}  

\begin{remark} To prove the Shapiro--Shields theorem, one can use that
$$
\Bigl( \prod_{j=1}^n \phi_{z_j,m}(z) \Bigr) \mathcal{D} \subset \mathcal{M}(\{z_1, \ldots, z_n\}, \mathcal{D}).
$$

Since 
$$
\psi_n(z) = \prod_{j=1}^n \phi_{z_j,m}(z) = \prod_{j=1}^n \Bigl( 1-\frac{1}{k^{m}_{z_j}(z_j)} \Bigr)^{-1/2} \Bigl( 1 - \frac{k^{m}_{z_j}(z)}{k^{m}_{z_j}(z_j)} \Bigr),
$$
we obtain that
$$
\psi_n(0) = \prod_{j=1}^{n}   \Bigl( 1 - \frac{1}{k^{m}_{z_j}(z_j)} \Bigr)^{1/2}. 
$$

Thus, the Shapiro--Shields condition,  
\begin{equation}
\sum_{n\ge 1} {1}/{k^{m}_{z_n}(z_n)} < \infty, 
\label{f25}
\end{equation}
is a necessary and sufficient condition for the sequence \( (\psi_n(0))_{n\ge 1} \) to have a non-zero limit. In this situation, the functions \( \psi_n \), the products of extremal functions $\phi_{z_j,m}$ for the Dirichlet space, converge to a non-zero function vanishing on \( \mathcal{Z} = (z_n)_{n} \).  

Since  
$$
k^\mu_z(z) = \sup\bigl\{|f(z)|^2 : f \in \mathcal{D}(\mu), \|f\|_\mu \le 1\bigr\},
$$
where $d\mu(\zeta) = |\varphi(\zeta)|^2 |d\zeta|$, $\|\varphi\|_\infty \leq 1$, and $\mathcal{D} = \mathcal{D}(m) \subset \mathcal{D}(\mu)$, it follows that  
\[
k^m_z(z) \leq k^{\mu}_z(z),\qquad z\in\mathbb D.
\]
Thus, the Shapiro--Shields condition \eqref{f25} implies condition \eqref{SScond}.

\end{remark}
\section{Sufficient conditions
}

Let $f \in H^2$ and let $B$ be a Blaschke product. 
Then, by Carleson's formula \cite{C1} (see also \cite[Theorem 4.1.3]{EKMR}),
\[
\mathcal{D}(Bf)
= \mathcal{D}(f)
+ \sum_{z\in \mathcal{Z}} \int_\mathbb{T} \frac{1 - |z|^2}{|1 - \overline{\zeta} z|^2} |f(\zeta)|^2 \,\frac{|d\zeta|}{2\pi},
\]
where $\mathcal{Z}$ is the zero set of $B$.  

Given $\alpha>0$ and $\lambda=(1-\delta_\lambda)e^{i\theta_\lambda}\in\mathbb{D}\setminus\{0\}$, we set
\begin{align}
r_{\lambda,\alpha}&= \delta_\lambda \log^{\alpha} \frac{1}{\delta_\lambda},\notag\\
\mathcal{S}_{\lambda,\alpha}
&= D(e^{i\theta_\lambda},r_{\lambda,\alpha})\cap  \mathbb{D},\notag\\
J_{\lambda,\alpha}
&= D(e^{i\theta_\lambda},r_{\lambda,\alpha})\cap  \mathbb{T},\label{dop18}
\end{align}
where $D(w,s)$ is the disk of radius $s$ centered at $w$.

The following result shows that if a Blaschke sequence admits a suitable covering by subregions associated with a zero set of the Dirichlet space, and satisfies a counting 
condition, then it is itself a zero set for the Dirichlet space.
\begin{theorem}
Let $\alpha>0$, $M<\infty$, let $\mathcal{Z}$ be a Blaschke sequence, and let $\Lambda$ be a zero set of $\mathcal{D}$.  
If there exist sets $\Delta_\lambda \subset \mathcal{S}_{\lambda,\alpha}$ such that  
\[
\mathcal{Z} \subset \bigcup_{\lambda\in\Lambda} \Delta_\lambda
\]
and
\begin{equation}\label{SSG11}
\sum_{\lambda \in \Lambda} \frac{\#(\mathcal{Z}\cap \Delta_\lambda)}
     {\log^M \frac{1}{1-|\lambda|}}
< \infty ,
\end{equation}
then $\mathcal{Z}$ is a zero set of $\mathcal{D}$.
\end{theorem}

\begin{proof} Without loss of generality, we assume that $\Lambda\cap D(0,1/2)=\emptyset$. 
Let $\varphi$ be an extremal function for $\mathcal{D}$ such that 
\[
\mathcal{M}(\Lambda,\mathcal{D})
= \varphi\, \mathcal{D}(\mu_\varphi),
\qquad 
d\mu_\varphi = |\varphi|^2 \, dm.
\]

For every $\lambda\in\Lambda$, let $\varphi_\lambda = \phi_{\lambda,m}$ be an extremal function vanishing at $\lambda$, and let  
$\phi_\lambda \in \mathcal{D}(\mu_\lambda)$ be an extremal function vanishing on $\Lambda\setminus\{\lambda\}$, with  
$d\mu_\lambda(\zeta) = |\varphi_\lambda(\zeta)|^2 dm$.  
We have 
$$
|\varphi|=|\varphi_\lambda\phi_\lambda|, \qquad \lambda\in\Lambda.
$$

Let $B$ be the Blaschke product with the zero set $\mathcal{Z}$, and choose $N\ge 1$ such that $2N > M$.  
Using Carleson's formula, we obtain that 
\begin{equation}\label{ineq1}
\mathcal{D}(B\varphi^N)
= \mathcal{D}(\varphi^N)
+ \frac{1}{2\pi} \sum_{z\in \mathcal{Z}} \int_\mathbb{T}
     \frac{1 - |z|^2}{|1 - \overline{\zeta} z|^2}
     |\varphi(\zeta)|^{2N} \, |d\zeta|.
\end{equation}

Recall that
\begin{equation}\label{equationextremal}
\varphi_\lambda(z)
= \Big(1 - \frac{1}{k_{\lambda}^m(\lambda)}\Big)^{-1/2}
  \Big(1 - \frac{k_{\lambda}^m(z)}{k_{\lambda}^m(\lambda)}\Big),  
\end{equation}
where $k_\lambda^m$ is the reproducing kernel given by \eqref{kerneldirichet}.  

Next,
$$
k_{\lambda}^m(\lambda)=\frac1{|\lambda|^2}\log \frac1{1-|\lambda|^2}\asymp \log \frac1{\delta_\lambda}.
$$

Let $\gamma >0$, $\lambda\in\mathbb D\setminus\{0\}$, $\zeta \in J_{\lambda,\gamma}$. Then we have 
\begin{multline}
|k_{\lambda}^m(\lambda)-k_{\lambda}^m(\zeta)|=\Bigl| \frac1{\overline{\lambda}\zeta}\log \frac1{1-\overline{\lambda}\zeta}-\frac1{|\lambda|^2}\log \frac1{1-|\lambda|^2} \Bigr|\\ \le 
\frac1{|\lambda|}\Bigl( \Bigl| \frac1\zeta-\frac1\lambda \Bigr|\log \frac1{1-|\lambda|^2} +\Bigl| \log \frac{1-|\lambda|^2}{1-\overline{\lambda}\zeta}\Bigr|  \Bigr)\\ \le 
\frac{|\zeta-\lambda|}{|\lambda|^2}\log \frac1{1-|\lambda|^2}+
\frac1{|\lambda|}\Bigl( \bigl| \zeta e^{-i\theta_\lambda}\bigr| + \Bigr|\log \frac{1-|\lambda|^2}{(\lambda-\zeta)e^{-i\theta_\lambda}}\Bigr|  \Bigr)\\ \le 
C\Bigl(\delta_\lambda\log^{\gamma+1}\frac1{\delta_\lambda} + \log\log \frac1{\delta_\lambda}\Bigr),
\label{eqextremal}
\end{multline} 
for an absolute constant $C$.

Therefore,
$$|\varphi_n(\zeta)|
\le
C_\gamma  \frac{\log \log (1/\delta_\lambda)}
       { \log (1/\delta_\lambda)},\qquad \zeta \in J_{\lambda,\gamma},
$$
with $C_\gamma$ depending only on $\gamma$. 

Now, since $2N > M$, we have 
\begin{equation}\label{ineq2}
\sum_{z \in \Delta_{\lambda}\cap\mathcal Z}
\int_{J_{\lambda,\gamma}}
  \frac{1 - |z|^2}{|1 - \overline{\zeta} z|^2}
  |\varphi(\zeta)|^{2N} \, |d\zeta|
\le
C'_\gamma \,
\frac{\#(\mathcal{Z}\cap \Delta_{\lambda})}
     {\log^M \frac{1}{1 - |\lambda|}}, \qquad \lambda\in\mathbb D\setminus\{0\},
\end{equation}
with $C'_\gamma$ depending only on $\gamma$. 

On the other hand, if $\gamma>\alpha>0$, $\lambda\in\mathbb D\setminus D(0,c(\alpha,\gamma))$, $\zeta\in \mathbb T\setminus J_{\lambda,\gamma}$, and $z \in \mathcal{S}_{\lambda,\alpha}$, then
\[
|1 - \overline{\zeta} z|\ge \frac{1}{2}|\zeta - e^{i\theta_\lambda}|.
\]
Therefore,
\begin{equation}\label{eqextremal1}
\int_{\mathbb{T}\setminus J_{\lambda,\gamma}}
  \frac{|d\zeta|}{|1 - \overline{\zeta} z|^2}\le
\frac{C_{\alpha,\gamma}}
     {(1-|\lambda|)
       \log^\gamma \frac{1}{1 - |\lambda|}},\qquad  z \in \mathcal{S}_{\lambda,\alpha},
\end{equation}
with $C_{\alpha,\gamma}$ depending only on $\alpha$ and $\gamma$.

Let now $\gamma>\alpha+M$. Since $\|\varphi\|_\infty \le 1$ and  
$$
1-|z| \lesssim (1 - |\lambda|)   \log^\alpha \!\frac{1}{1 - |\lambda|},\qquad  z \in \mathcal{S}_{\lambda,\alpha}, 
$$   
we obtain that    
\begin{align}\label{ineq3}
\sum_{z \in \Delta_{\lambda}\cap\mathcal Z}
\int_{\mathbb{T}\setminus J_{\lambda,\gamma}}
  \frac{1 - |z|^2}{|1 - \overline{\zeta} z|^2}  |\varphi(\zeta)|^{2N}\, |d\zeta|
\lesssim 
\frac{\#(\mathcal{Z}\cap \Delta_{\lambda})}
     {\log^{\gamma - \alpha} \frac{1}{1 - |\lambda|}} .
\end{align}

Combining \eqref{ineq1}, \eqref{ineq2}, \eqref{ineq3}, and the assumption \eqref{SSG11}, we obtain that
\[
\mathcal{D}(B\varphi^{N})
\lesssim
\sum_{\lambda\in\Lambda}
\frac{\#(\mathcal{Z}\cap \Delta_{\lambda})}
     {\log^M \frac{1}{1 - |\lambda|}}
< \infty.
\]
This concludes the proof.
\end{proof}

Let $\mathcal{Z} = \{z_n\}_{n\ge 1} \subset \mathbb{D}$, and let $p=(p_n)_{n\ge 1}$ be a sequence of positive integers.  
We say that $(\mathcal{Z},p)$ is a zero divisor for $\mathcal{D}$ if there exists a non-vanishing function $f\in\mathcal{D}$ such that
\[
f^{(k)}(z_n)=0 \qquad 0\le k<p_n,\,n\ge 1.
\]

If $\mathcal{Z}$ is a zero set for $\mathcal{D}$, then (by Carleson's formula) there exists $f\in\mathcal{D}$ such that
\[
\sum_{n\ge 1}
\int_\mathbb{T}
  \frac{1 - |z_n|^2}{|1 - \bar{\zeta} z_n|^2}
  |f(\zeta)|^2
\, \frac{|d\zeta|}{2\pi}
< \infty.
\]
Hence, we may choose a sequence $p=(p_n)_{n\ge 1}$ with $p_n\to\infty$ such that
\[
\sum_{n\ge 1}
p_n\!
\int_\mathbb{T}
  \frac{1 - |z_n|^2}{|1 - \bar{\zeta} z_n|^2}
  |f(\zeta)|^2
\, \frac{|d\zeta|}{2\pi}
< \infty,
\]
and, hence, $(\mathcal{Z},p)$ is a zero divisor for $\mathcal{D}$, with multiplicities tending to infinity.  

The next corollary establishes a quantative version of this result, giving 
a convenient sufficient condition for zero divisors.

\begin{corollary}
Let $\mathcal{Z}=\{z_n\}_{n\ge 1}$ be a zero set for $\mathcal{D}$, separated  in the pseudohyperbolic metric, and let $p=(p_n)_{n\ge 1}$ be a sequence of positive integers.   
If 
\[
\sum_n
\frac{p_n}
     {\log^M \frac{1}{1 - |z_n|}}
< \infty,
\]
for some $M<\infty$, then $(\mathcal{Z},p)$ is a zero divisor of $\mathcal{D}$.
\end{corollary}

If we start with a zero sequence satisfying a condition stronger than that of Shapiro--Shields, we obtain the following result.

\begin{theorem}
Let $\beta\in(0,1)$ and let $\Lambda=\{\lambda_n\}\subset \mathbb{D}$ be such that 
\begin{equation}
\label{appD}
\sum_{\lambda\in\Lambda} \frac{1}{\log^\beta \frac{1}{1-|\lambda|}} < \infty.
\end{equation}
Let $\alpha>0$, $\gamma\in(0,1-\beta)$, let $\Delta_\lambda\subset \mathcal{S}_{\lambda,\alpha}$, $\lambda\in\Lambda$, 
and let $\mathcal{Z}$ be a subset of $\bigcup_{\lambda\in\Lambda}\Delta_\lambda$  
such that 
$$
\sum_{\lambda\in\Lambda} \#(\mathcal{Z}\cap\Delta_\lambda)\,
e^{-\log^\gamma \frac{1}{1-|\lambda|}} < \infty.
$$
Then $\mathcal{Z}$ is a zero set of $\mathcal{D}$.
\end{theorem}

\begin{proof}
Let us first observe that $\mathcal{Z}$ is a Blaschke sequence. Indeed, if 
$z\in\mathcal{S}_{\lambda,\alpha}$, then
\[
\log^\gamma\frac{1}{1-|\lambda|} \le \log\frac{1}{1-|z|}.
\]
Hence,
\[
\sum_{\lambda\in\Lambda} 
\#(\mathcal{Z}\cap\Delta_\lambda)\,
e^{-\log^\gamma \frac{1}{1-|\lambda|}}
\ge
\sum_{z\in\mathcal{Z}} (1-|z|).
\]

Next, let $\varphi_\lambda$ be the extremal function vanishing at $\lambda\in\Lambda$, given by 
\eqref{equationextremal}, and set
$$
\Phi=\prod_{\lambda\in\Lambda} \varphi_\lambda^{b_\lambda}, 
$$
where $b_\lambda$ is the integer part of $\log^\gamma\frac{1}{1-|\lambda|}$. 
Since $\gamma\in(0,1-\beta)$, condition \eqref{appD} ensures that the product defining 
$\Phi$ converges and $\Phi\in\mathcal{D}$, as a product of multipliers.

Let $B$ be the Blaschke product associated with $\mathcal{Z}$. It remains to show that 
$B\Phi\in\mathcal{D}$.  
By Carleson's formula, it suffices to verify that 
\[
\sum_{z\in\mathcal{Z}}
\int_{\mathbb{T}} 
\frac{1-|z|^2}{|1-\overline{\zeta} z|^2}\,
|\Phi(\zeta)|^2\,
\,\frac{d\zeta}{2\pi}<\infty.
\]

We set
$$J_\lambda = J_{\lambda,b_\lambda},
$$
with $J_{\cdot,\cdot}$ defined in \eqref{dop18}.

If $\lambda\in\Lambda$, $\delta_\lambda=1-|\lambda|$, and $\zeta\in J_\lambda$, then by \eqref{eqextremal}, we have 
$$
|\Phi(\zeta)|
\le 
|\varphi_{\lambda}(\zeta)|^{b_\lambda}
\le 
\left(
C\delta_\lambda\log^{b_\lambda}\frac1{\delta_\lambda} + \frac{C\log\log (1/\delta_\lambda)}{\log(1/\delta_\lambda)}\Bigr)
\right)^{b_\lambda}
\le
C_1\exp\!\Big(
-\log^\gamma\frac{1}{1-|\lambda|}
\Big).
$$
for some absolute constant $C_1$. 

Therefore, if $z\in\Delta_{\lambda}$, then 
\[
\int_{J_\lambda}
\frac{1-|z|^2}{|1-\overline{\zeta} z|^2}
|\Phi(\zeta)|^2
\,\frac{d\zeta}{2\pi}
\le
C_1\exp\!\Big(
-\log^\gamma\frac{1}{1-|\lambda|}
\Big).
\]

Furthermore, 
as in \eqref{eqextremal1}, we have 
\[
\int_{\mathbb{T}\setminus J_\lambda}
\frac{1-|z|^2}{|1-\overline{\zeta} z|^2}
|\Phi(\zeta)|^2
\,\frac{d\zeta}{2\pi}
\lesssim 
\frac{1-|z|}
{(1-|\lambda|)\big(\log \frac{1}{1-|\lambda|}\big)^{b_\lambda}}.
\]

Thus,
\begin{multline*}
\sum_{z\in\mathcal{Z}}
\int_{\mathbb{T}}
\frac{1-|z|^2}{|1-\overline{\zeta} z|^2}
|\Phi(\zeta)|^2
\,\frac{d\zeta}{2\pi}
\\
\lesssim  
\sum_{\lambda\in\Lambda}
\#(\mathcal{Z}\cap\Delta_{\lambda})
\Big(
e^{-\log^\gamma\frac{1}{1-|\lambda|}}
+
e^{-(b_\lambda-\alpha)\log\log\frac{1}{1-|\lambda|}}
\Big)
\\
\lesssim 
\sum_{\lambda\in\Lambda}
\#(\mathcal{Z}\cap\Delta_{\lambda})
\,
e^{-\log^\gamma\frac{1}{1-|\lambda|}}
< \infty.
\end{multline*}

This proves that $B\Phi\in\mathcal{D}$ and completes the proof.
\end{proof}

\subsection*{Acknowledgements}
The authors are grateful to Anton Baranov for a helpful remark. 
The first and third authors were supported by ANR Project ANR-24-CE40-5470. 
The second author was partially supported by the Arab Fund Foundation Fellowship Program, The Distinguished Scholar Award -- File 1092. He also acknowledges the Laboratory of Analysis and Applied Mathematics (LAMA) at Gustave
Eiffel University for its kind hospitality during the preparation of this paper.

\end{document}